\documentclass[12pt]{article}
\usepackage{amssymb}
\usepackage{amsmath}
\usepackage{amsfonts,amsthm,euscript,relsize}
\usepackage[english]{babel}
\usepackage[colorlinks=true,linkcolor=blue,urlcolor=blue,citecolor=blue]{hyperref}
\usepackage[pagewise]{lineno}
\usepackage{fancyhdr}
\numberwithin{equation}{section}
\newcommand{\esssup}{{\mathrm{ess}\sup\,}}
\newcommand{\essinf}{{\mathrm{ess}\inf\,}}

\begin{document}


\title{\bf Minimization and Steiner symmetry of the first eigenvalue for a fractional eigenvalue problem
with indefinite weight}
\date{}
\maketitle
\begin{center}
\textbf {Claudia Anedda, Fabrizio Cuccu and Silvia Frassu}
\\Mathematics and Computer Science Department,\\ University of Cagliari, 09124 Cagliari (CA), Italy
\end{center}

\begin{abstract}
 \noindent Let $\Omega\subset\mathbb{R}^N$, $N\geq 2$, be an open bounded connected set. We consider the fractional weighted eigenvalue problem 
 $(-\Delta)^s u =\lambda \rho u$ in $\Omega$ with homogeneous Dirichlet boundary condition, 
 where $(-\Delta)^s$, $s\in (0,1)$, is the fractional Laplacian operator, $\lambda \in \mathbb{R}$ and 
$ \rho\in L^\infty(\Omega)$.\\
We study weak* continuity, convexity and G\^ateaux differentiability of the map 
$\rho\mapsto1/\lambda_1(\rho)$, where
$\lambda_1(\rho)$ is the first positive eigenvalue.
Moreover, denoting by $\mathcal{G}(\rho_0)$ the class of rearrangements of $\rho_0$,
we prove the existence of a minimizer of $\lambda_1(\rho)$ when $\rho$ varies on $\mathcal{G}(\rho_0)$.
Finally, we show that, if $\Omega$ is Steiner symmetric, then every minimizer shares the same symmetry. 
\end{abstract}

\let\thefootnote\relax\footnote{{\bf Keywords:} fractional Laplacian, eigenvalue problem, optimization, 
Steiner symmetry.}

\let\thefootnote\relax\footnote{{\bf AMS (2010) Subject Classifications:} 35R11, 47A75, 49R05.}

\newtheorem{theorem}{Theorem}[section]
\newtheorem{proposition}[theorem]{Proposition}
\newtheorem{lemma}[theorem]{Lemma}
\newtheorem{corollary}[theorem]{Corollary}
\theoremstyle{definition}
\newtheorem{definition}{Definition}[section]
\newtheorem{rem}{Remark}[section]

\section{Introduction}
Recently, great attention has been focused on the study of fractional and nonlocal operators of elliptic type, both for pure mathematical research and in view of concrete real-world applications.
This type of operators arises in a quite natural way in many different contexts, such as, among others, the thin obstacle problem, optimization, finance, phase transition, stratified materials, anomalous diffusion, crystal dislocation, soft thin films, semipermeable membranes, flame propagation, conservation laws, ultrarelativistic limits of quantum mechanics, quasi-geostrophic flows, multiple scattering, minimal surfaces, materials science, water waves, chemical reactions of liquids, population dynamics, geophysical fluid dynamics and mathematical finance. In all these cases, the nonlocal effect is modeled by the singularity at infinity.
For more details and applications, see \cite{BV, LC, V} and the references therein.\\
In this paper we consider the weighted fractional eigenvalue problem
\begin{equation}\label{p0}
\begin{cases}(-\Delta)^s u =\lambda \rho u \quad &\text{in } \Omega\\
u=0 &\text{in } \Omega^c, \end{cases}  
\end{equation}
where $\Omega\subset\mathbb{R}^N$ ($N\geq 2$) is a bounded domain with $C^{2}$ boundary $\partial\Omega$ and
$(-\Delta)^s$, $s\in (0,1)$, denotes the fractional Laplacian operator defined for all $x\in \mathbb{R}^N$ by 
\begin{equation*}
(-\Delta)^s u(x)= C(N,s) \lim_{\varepsilon \to 0^+} \int_{\mathbb{R}^N\smallsetminus B_\varepsilon(x)}
\frac{u(x)-u(y)}{|x-y|^{N+2s}}\,  dy,
\end{equation*}
where $u:\mathbb{R}^N\to \mathbb{R}$ is a Lebesgue measurable function and $ C(N,s)$ is 
a suitable normalization constant. In the sequel we will assume
$C(N,s) = 1$ (for a precise evaluation of $C(N,s)$ see \cite{CS, DNPV}). 
Finally, $\rho\in L^\infty(\Omega)$, $\lambda\in\mathbb{R}$ and $\Omega^c=\mathbb{R}^N
\smallsetminus \Omega$.\\
The operator $(-\Delta)^s$ is nonlocal, in the sense that the value of $(-\Delta)^s u(x)$ at any point $x\in \Omega$ depends not only on the values of $u$ on the whole $\Omega$, but actually on the whole $\mathbb{R}^n$, since $u(x)$ 
can be thought as the expected value of a random variable tied to a process randomly jumping arbitrarily far from the point $x$. In this sense, the natural Dirichlet boundary condition consists in assigning the values of $u$ in $ \Omega^c$ rather than merely on $\partial \Omega$ (a general reference on the theory can be found in \cite{DNPV, MBRS}). \\
 The problem \eqref{p0} with $\rho\equiv 1$ has been investigated by Servadei and Valdinoci in \cite{SV2} for a general nonlocal operator. Molica Bisci et al., in \cite{MBRS}, studied the same problem with a positive and Lipschitz continuous weight $\rho$.
 Iannizzotto and Papageorgiou, in \cite{IP}, considered the case of a general positive function $\rho\in L^\infty(\Omega)$ and
 Frassu and Iannizzotto, in \cite{FI}, treated a more general eigenvalue problem with indefinite weight $\rho\in L^\infty(\Omega)$.\\
 We denote by $\lambda_k(\rho)$, $k\in \mathbb{Z}\smallsetminus 
 \{0\}$, the $k$-th eigenvalue of problem
 \eqref{p0} corresponding to the weight $\rho$.
 In this paper we study the dependence of $\lambda_k(\rho)$ on $\rho$, in 
 particular we investigate
 continuity and, for $k=1$, convexity and differentiability properties. Then, we examine the minimization of $\lambda_1(\rho)$ in
 the class of rearrangements $\mathcal{G}(\rho_0)$ of a fixed function $\rho_0\in L^\infty(\Omega)$.
 We prove the existence of minimizers and a characterization of them in terms of the
 eigenfunctions relative to $\lambda_1(\rho)$.
 Moreover, when $\Omega$ is a Steiner symmetric domain, we get that any minimizer inherits the same symmetry. Consequently, if $\Omega $ is a ball, there exists a unique radially symmetric minimizer.\\
The analogous problem in the case of the Laplacian operator has been studied by Cox and McLaughlin
in \cite{Cox1,Cox2} when the weight $\rho_0$ is a positive step function. 
Cosner et al. in \cite{CCP} studied the same optimization problem with an indefinite weight $\rho_0
\in L^\infty(\Omega)$ for the first eigenvalue, they proved existence of optimizers and a characterization formula of them. Related problems are
investigated in \cite{ClaFa1, ClaFa2}. In the first paper the eigenvalue problem is driven by the $p$-Laplacian
operator. In the second an example of symmetry breaking of the minimizer is exhibited.
For a complete survey on the optimization of eigenvalues 
related to elliptic problems we refer the reader to \cite{He}. \\
We remark that the argument used in this paper to prove the existence of minimizers is inspired by the approach 
of \cite{He} and it is different from those used in \cite{CCP,Cox1,Cox2},
 nevertheless it can be applied also for the corresponding problem driven by the Laplacian operator.\\
This paper is organized in this way: in Section 2 we fix the functional framework and study the
eigenvalues of problem \eqref{p0}; in Section 3 we collect some results about rearrangements of measurable functions;
in Section 4 we prove the existence results; finally, in Section 5 we focus on the symmetry of the minimizers.\\
Throughout the paper, and unless otherwise specified, measurable means Lebesgue measurable and $|E|$ denotes the Lebesgue measure of a measurable set $E\subseteq \mathbb{R}^N$.

\section{Fractional weighted eigenvalue problem}

 Let $\Omega\subseteq \mathbb{R}^N$, $N\geq 2$, be a bounded domain with $C^2$ boundary. We denote by
\begin{equation*}
\langle u,v \rangle_{L^2(\Omega)}=\int_\Omega uv \, dx \quad \forall\, u, v\in L^2(\Omega)
\end{equation*}
the usual inner product in $L^2(\Omega)$
and by
\begin{equation*}
\|u\|_{L^2(\Omega)}^2=\langle u,u \rangle_{L^2(\Omega)} \quad \forall\, u\in L^2(\Omega)
\end{equation*}
the corresponding norm.\\
\medskip

In order to formulate problem \eqref{p0} in weak form we introduce the fractional Sobolev space (for a systematic treatise of this topics see \cite{DNPV}). 
For any $s\in (0,1)$ we define the fractional Sobolev space
$$H^s(\mathbb{R}^N) =\left\{u\in L^2(\mathbb{R}^N) : \int_{\mathbb{R}^{2N}}
\frac{\big(u(x)-u(y) \big)^2}{|x-y|^{N+2s}}\, dx \,dy<\infty\right\}$$
and its subspace
$$H_0^s(\Omega)=\{u\in H^s(\mathbb{R}^N): u(x)=0 \text{ for a.e. }x\in \Omega^c\}.$$
$H_0^s(\Omega)$ is a 
separable Hilbert space under the
inner product
$$\langle u,v\rangle_{H_0^s(\Omega)}=  \int_{\mathbb{R}^{2N}}
\frac{\big(u(x)-u(y) \big)\big(v(x)-v(y) \big)}{|x-y|^{N+2s}}\, dx\, dy$$
whose associated norm is
$$\|u\|_{H_0^s(\Omega)}^2=  \int_{\mathbb{R}^{2N}}
\frac{\big(u(x)-u(y) \big)^2}{|x-y|^{N+2s}}\, dx\, dy.$$ 
We denote by
$H^{-s}(\Omega)$ and $\| \cdot \|_{H^{-s}(\Omega)}$ the topological dual of 
$H_0^s(\Omega)$ and its norm.
Clearly $H_0^s(\Omega)\subset L^2(\Omega)$ and moreover $H_0^s(\Omega)$
contains $C^\infty_0(\Omega)$ as a dense subset.\\
As in the case of the usual Sobolev spaces, the following inclusions 
\begin{equation*}i: H_0^s(\Omega) \hookrightarrow L^2(\Omega)
\end{equation*}
 \begin{equation*}j:L^2(\Omega)
 \hookrightarrow H^{-s}(\Omega)\end{equation*}
are compact and dense and there exists a positive constant $C$ such that 
\begin{equation}\label{i}\|\varphi\|_{L^2(\Omega)}\le C \|\varphi\|_{H_0^s(\Omega)} \quad \forall\, \varphi\in H_0^s(\Omega)\end{equation}
 \begin{equation}\label{j}\|\varphi\|_{H^{-s}(\Omega)}\le C \|\varphi\|_{L^2(\Omega)} \quad \forall\, \varphi\in L^2(\Omega).\end{equation}
\\
 \medskip

Let us now introduce the notion of weak solution of the boundary value problem
\begin{equation}\label{p2}
\begin{cases}(-\Delta)^s u = f \quad &\text{in } \Omega\\
u=0 &\text{in } \Omega^c, \end{cases}  
\end{equation}
where $f\in H^{-s}(\Omega)$. A function $u\in H_0^s(\Omega)$ is called
\emph{weak solution} of problem \eqref{p2} if
\begin{equation*}
\langle u,\varphi\rangle_{H_0^s(\Omega)}
=\langle f,\varphi\rangle \quad\forall\, \varphi\in H_0^s(\Omega)
\end{equation*}
holds, where $\langle f,g\rangle$ means the duality between $f\in H^{-s}(\Omega)$ and $g\in H_0^s(\Omega)$.
By the Riesz-Fr\'echet representation Theorem, for every $f\in H^{-s}(\Omega)$ 
there exists a unique solution $u\in H_0^{s}(\Omega)$ of \eqref{p2} and moreover 
\begin{equation}\label{normadu}
\|u\|_{H_0^{s}(\Omega)}= \|f\|_{H^{-s}(\Omega)}.
\end{equation}

We call $G$,
\begin{equation}\label{operatore}G:H^{-s}(\Omega)\to H_0^s(\Omega),\end{equation} the linear operator defined by $G(f)=u$.
Identity \eqref{normadu} implies 
\begin{equation}\label{normaduG}
\|G\|_{\mathcal{L}(H^{-s}(\Omega), H_0^{s}(\Omega))}=1.
\end{equation}

For any $\rho$ in  $L^\infty (\Omega)$, let $M_\rho: L^2(\Omega) \to L^2(\Omega)$ be the linear operator 
defined by $M_\rho(f)= \rho f$. Of course
\begin{equation}\label{mro}\|\rho f\|_{L^2(\Omega)} \leq 
\|\rho\|_{L^\infty(\Omega)}\|f\|_{L^2(\Omega)}.\end{equation}
Next, we introduce the linear operator
\begin{equation}\label{G}
G_\rho:H_0^s(\Omega)\to H_0^s(\Omega)
\end{equation}
defined by $G_\rho=G\circ j \circ M_\rho \circ i $ or, briefly, $G_\rho (f)=G(\rho f)$. Equivalently, $u=G_\rho (f)$ is the unique 
weak solution of the problem
\begin{equation*}
\begin{cases}(-\Delta)^s u = \rho f \quad &\text{in } \Omega\\
u=0 &\text{in } \Omega^c,\end{cases}  
\end{equation*}
i.e. 
\begin{equation}\label{vespa}
\langle u,\varphi\rangle_{H_0^s(\Omega)}
=\langle \rho f,\varphi\rangle_{L^2(\Omega)} \quad\forall\, \varphi\in H_0^s(\Omega).
\end{equation}

From  \eqref{i}, \eqref{j}, \eqref{normaduG} and \eqref{mro} it follows straightforwardly that 
\begin{equation*}
\|G_\rho\|_{\mathcal{L}(H_0^{s}(\Omega), H_0^{s}(\Omega))} \leq C^2 
\|\rho\|_{L^\infty(\Omega)}.
\end{equation*}
In the sequel we will use the formula 
\begin{equation}\label{linearit}
G_{a\rho +b\eta}= a\,  G_{\rho}+b\, G_\eta\quad \forall\, \rho, \eta\in L^\infty(\Omega), \ 
\forall\, a, b\in \mathbb{R}.
\end{equation}
In particular, \eqref{linearit} implies $G_{-\rho}=-G_\rho$ for all $\rho\in L^\infty(\Omega)$.

\begin{proposition}\label{self}
Let $G_\rho$ be the operator \eqref{G}. Then $G_\rho$ is a self-adjoint compact operator.
\end{proposition}

\begin{proof}
For all $f, g\in H_0^s(\Omega)$, by \eqref{vespa}, we have
$$ \langle G_\rho(f), g\rangle_{H_0^s(\Omega)} = 
\langle G(\rho f), g\rangle_{H_0^s(\Omega)}= 
\langle \rho f, g\rangle_{L^2(\Omega)}=
\langle \rho g, f\rangle_{L^2(\Omega)} =
\langle G_\rho (g), f\rangle_{H_0^s(\Omega)},$$
then $G_\rho$ is self-adjoint.\\
The compactness of the operator $G_\rho$ is an immediate consequence of the 
representation $G_\rho=G\circ j \circ M_\rho \circ i $ and the compactness of
$i$ and $j$.
\end{proof}

By general theory of self-adjoint compact operators (see \cite{Br,DF,Lax})
it follows that all nonzero eigenvalues of $G_\rho$ have a finite dimensional eigenspace 
and they can be obtained by the Fischer's Principle
\begin{equation}\label{Fischer1}\mu_k(\rho) =  \max_{F_k}
\min_{f\in F_k\atop f\neq 0}
\cfrac{\langle G_\rho f, f \rangle_{H^s_0(\Omega)} }{\|f\|^2_{H_0^{s}(\Omega)}} \,= \max_{F_k}
\min_{f\in F_k\atop f\neq 0}
\cfrac{\int_\Omega \rho f^2 \, dx }{\|f\|^2_{H_0^{s}(\Omega)}} \,, \quad k=1, 2, 3, \ldots\end{equation}  
and
\begin{equation*}\mu_{-k}(\rho) =  \min_{F_k}
\max_{f\in F_k\atop f\neq 0}
\cfrac{\langle G_\rho f, f \rangle_{H^s_0(\Omega)} }{\|f\|^2_{H_0^{s}(\Omega)}} \,=\min_{F_k}
\max_{f\in F_k\atop f\neq 0}
\cfrac{\int_\Omega\rho f^2 \, dx }{\|f\|^2_{H_0^{s}(\Omega)}}\,,\quad k=1, 2, 3,  \ldots,\end{equation*}  
where the first extrema are taken over all the subspaces $F_k$ of $H_0^{s}(\Omega)$ of dimension
$k$. The sequence $\{\mu_k(\rho)\}$ contains all the real positive
eigenvalues (repeated with their multiplicity), is decreasing and converging to zero, whereas $\{\mu_{-k}(\rho)\}$ 
is formed by all the real negative
eigenvalues (repeated with their multiplicity), is increasing and converging to zero.\\

\begin{rem}\label{osserv} By the Fischer's Principle it follows easily that
$\mu_{-k}(\rho)= -\mu_{k}(-\rho)$ for all $\rho\in L^\infty(\Omega)$ and $k=1, 2, 3, \ldots$\\
For this reason, in the rest of the paper, we will consider mainly positive eigenvalues.
\end{rem}

We will write $\{\rho>0\}$ as short form of $\{x\in \Omega: \rho(x)>0\}$ and similarly 
$\{\rho<0\}$ for $\{x\in \Omega: \rho(x)<0\}$. The following proposition
is analogous to \cite[Proposition 1.11]{DF}.

\begin{proposition}\label{segnorho}Let $\rho \in L^\infty(\Omega)$, $G_\rho$ the operator defined in \eqref{G}
and $\mu_k(\rho)$, $\mu_{-k}(\rho)$ its eigenvalues. The following statements hold:\\
i) if $|\{\rho >0\}|= 0$, then there are no positive eigenvalues;\\
ii)  if $|\{\rho >0\}|> 0$, then there is a sequence of positive eigenvalues $\mu_k(\rho)$;\\
iii) if $|\{\rho <0\}|= 0$, then there are no negative eigenvalues;\\
iv)  if $|\{\rho <0\}|> 0$, then there is a sequence of negative eigenvalues $\mu_{-k}(\rho)$.
\end{proposition}

\begin{proof}
i) Let $\mu$ be an eigenvalue and $u$ a corresponding eigenfunction. Then 
$$\mu= \cfrac{\langle G_\rho u, u\rangle_{H_0^{s}(\Omega)}}{\|u\|^2_{H_0^{s}(\Omega)}}\, 
= \cfrac{\int_\Omega \rho u^2\, dx }{\|u\|^2_{H_0^{s}(\Omega)}}\, \leq 0.$$
ii) By measure theory covering theorems, for each positive integer $k$ there exist $k$ 
disjoint closed balls $B_1, \ldots, B_k$ in $\Omega$ such that $| B_i \cap \{\rho>0\}|>0$ 
for $i=1, \ldots, k$. Let $f_i\in C^\infty_0(B_i)$ such that 
$\int_\Omega \rho f_i^2 \, dx=1$ for every $i=1, \ldots, k$. Note that the functions 
$f_i$ are linearly independent and let $F_k= \ $span$  \{f_1, \ldots, f_k\}$. $F_k$ is 
a subspace of $H_0^{s}(\Omega)$ and
for every $f\in F_k\smallsetminus 
\{0\}$, $f=\sum_{i=1}^k a_i f_i$, $a_i\in \mathbb{R}$,  we have
\begin{equation*} \begin{split}\cfrac{\langle G_\rho f, 
f\rangle_{H_0^{s}(\Omega)}}{\|f\|^2_{H_0^{s}(\Omega)}}\, = &
\cfrac{\int_\Omega \rho f^2  \, dx }{\|f\|^2_{H_0^{s}(\Omega)}} \, =
\cfrac{\sum_{i, j=1}^k \int_\Omega \rho f_i f_j a_i a_j \,dx}{\sum_{i,j=1}^k\langle f_i, f_j\rangle_{H_0^{s}(\Omega)} a_i a_j} \,\\&
=\cfrac{\sum_{i=1}^k a_i^2}{\sum_{i,j=1}^k\langle f_i, f_j\rangle_{H_0^{s}(\Omega)} a_i a_j} \,=
\cfrac{\|a\|^2_{\mathbb{R}^k}}{\langle E_k a, a\rangle_{\mathbb{R}^k}}\,\geq \cfrac{1}{\|E_k\|}\, >0,
\end{split}\end{equation*}
where $\|a\|_{\mathbb{R}^k}$, $\|E_k\|$ and $\langle E_k a, a\rangle_{\mathbb{R}^k}$ denote, respectively,
the euclidean norm of the vector $a=(a_1, \ldots, a_k)$, the norm of the non null matrix 
$E_k=\left( \langle f_i, f_j\rangle_{H_0^{s}(\Omega)} \right)_{i, j=1}^k$ and the inner product in $\mathbb{R}^k$.
From the Fischer's Principle \eqref{Fischer1} we conclude that $\mu_k(\rho)\geq \cfrac{1}{\|E_k\|}\, >0$ for
every $k$.\\
The cases iii) and iv) are similarly proved.
\end{proof}

Finally, we introduce the weak formulation of problem \eqref{p0}.
A function $u\in H_0^s(\Omega)\smallsetminus \{0\}$ is said an \emph{eigenfunction} of \eqref{p0} associated to the \emph{eigenvalue} $\lambda$ if 
\begin{equation}\label{rana}
\langle u,\varphi\rangle_{H_0^s(\Omega)}
=\lambda \langle \rho u,\varphi\rangle_{L^2(\Omega)} \quad\forall\, \varphi\in H_0^s(\Omega),
\end{equation}
that is
\begin{equation*}
\int_{\mathbb{R}^{2N}}\frac{\big(u(x)-u(y) \big)\big(\varphi(x)-\varphi(y) \big)}{|x-y|^{N+2s}}\, dx dy
= \lambda \int_\Omega \rho(x) u(x)\varphi(x)\, dx \quad\forall\, \varphi\in H_0^s(\Omega).
\end{equation*}

It is easy to check that zero is not an eigenvalue of problem \eqref{p0}.
The eigenvalues of problem \eqref{p0} are exactly the reciprocal 
of the nonzero eigenvalues of the operator $G_\rho$ and the correspondent eigenspaces coincide. 
Indeed, if $\lambda\neq 0$ is an eigenvalue of problem \eqref{p0}
and $u$ is an associated eigenfunction, by \eqref{rana} we have
\begin{equation*}
\left\langle \frac{u}{\lambda},\varphi\right\rangle_{H_0^s(\Omega)}
=\langle \rho u,\varphi\rangle_{L^2(\Omega)} \quad\forall\, \varphi\in H_0^s(\Omega)
\end{equation*}
and then, by definition of $G_\rho$, $G_\rho (u)= \cfrac{u}{\lambda}\, $. Consequently, in general,
the eigenvalues of problem \eqref{p0} form two monotone sequences
$$ 0<\lambda_1(\rho)\leq \lambda_2(\rho)\leq\ldots\leq  \lambda_k(\rho)\leq \ldots$$
and
$$ \ldots\leq\lambda_{-k}(\rho)\leq\ldots\leq \lambda_{-2}(\rho)\leq  \lambda_{-1}(\rho)<0  ,$$
where every eigenvalue appears as many times as its multiplicity, the latter being finite
owing to the compactness of $G_\rho$. \\
It has been recently shown in \cite{FI} that $\lambda_1(\rho)$ and $\lambda_{-1}(\rho)$ 
are simple and
any associated eigenfunction is one signed in $\Omega$.  We call \emph{first eigenfunction}
any eigenfunction relative to $\lambda_1(\rho)$. The 
variational characterization \eqref{Fischer1} for $k=1$ becomes
\begin{equation}\label{mu1}\mu_1(\rho) = \max_{f\in H_0^s(\Omega)\atop f\neq 0}
\cfrac{\int_\Omega \rho f^2  \, dx }{\|f\|^2_{H_0^{s}(\Omega)}} \,\end{equation}
and, thus, for $\lambda_1(\rho)$ we have
\begin{equation}\label{2a}\lambda_1(\rho) = \min_{u\in H_0^{s}(\Omega)\atop u\neq 0}
\cfrac{\|u\|^2_{H_0^{s}(\Omega) }}{\int_\Omega \rho u^2 \, dx}\,.\end{equation}  
The maximum in \eqref{mu1} (respectively the minimum in \eqref{2a})
is obtained if and only if $f$ (respectively $u$) is a first eigenfunction.  Throughout the paper 
we will denote by $u_\rho$ the first positive eigenfunction of problem \eqref{p0}
normalized by
\begin{equation}\label{normaliz1}
\|u_\rho\|_{H_0^{s}(\Omega)}= 1,
\end{equation}
which is equivalent to 
\begin{equation}\label{normaliz2}
\int_\Omega \rho u_\rho^2 \, dx= \cfrac{1}{\lambda_1(\rho)}\,.
\end{equation}
As last comment, we observe that $\mu_1(\rho)$ is homogeneous of degree 1, i.e. 
\begin{equation}\label{homo}\mu_1(\alpha\rho)= \alpha \mu_1(\rho) \quad \forall\, \alpha>0.\end{equation}
This follows immediately from \eqref{mu1}.
\medskip


\section{Rearrangements of measurable functions} In this section we introduce the concept of
rearrangement of a measurable function and summarize some related results we will use in next section. 
The idea of rearranging a function dates back to the book \cite{hardy52} 
of Hardy, Littlewood and P\'olya, since than many authors have investigated both extensions and 
applications of this notion. Here we relies on the results in \cite{Alvino,B,B89,day70,
Kaw,ryff67}.

Let $\Omega$ be an open bounded set of $\mathbb{R}^N$.

\begin{definition}
For every measurable 
function $f:\Omega\to\mathbb{R}$ the function $d_f:\mathbb{R}\to [0,|\Omega|]$ defined by
$$d_f(t)=|\{x\in\Omega: f(x)>t\}|$$
is called \emph{distribution function of $f$}.
\end{definition}

The symbol $\mu_f$ is also used. It is easy to prove the
following properties of $d_f$.

\begin{proposition}
For each $f$ the distribution function $d_f$ is decreasing, right continuous and
the following identities hold true
$$\lim_{t\to-\infty}d_f(t)=|\Omega|,\quad\quad \lim_{t\to\infty}d_f(t)=0.$$
\end{proposition}

\begin{definition}
Two measurable functions $f,g:\Omega \to \mathbb{R}$ are called \emph{equimeasurable} functions 
or \emph{rearrengements} of one another if one of the following equivalent conditions is satisfied

i)   $|\{x\in \Omega: f(x)>t\}|=|\{x\in \Omega: g(x)>t\}| \quad \forall\, t\in \mathbb{R}$;

ii)  $d_f=d_g$.
\end{definition}

Equimeasurability of $f$ and $g$ is denoted by $f\sim g$. Equimeasurable functions
share global extrema and integrals as it is stated precisely by the following proposition.
\begin{proposition}\label{rospo}
Suppose $f\sim g$ and let $F:\mathbb{R}\to\mathbb{R}$ be a Borel measurable function, then
 
i) $|f|\sim |g|$;

ii) $\esssup f=\esssup g$ and  $\essinf f=\essinf g$;

iii) $F\circ f\sim F\circ g$;

iv) $F\circ f\in L^1(\Omega)$ implies $F\circ g\in L^1(\Omega)$ and $\int_\Omega F\circ f\,dx=
\int_\Omega F\circ g\,dx$.

\end{proposition}
For a proof see, for example, \cite[Proposition 3.3]{day70} or \cite[Lemma 2.1]{B89}.

In particular, for each $1\leq p\leq\infty$, if $f\in L^p(\Omega)$ and $f\sim g$ then
$g\in L^p(\Omega)$ and 
\begin{equation}\label{equi}
 \|f\|_{L^p(\Omega)}=\|g\|_{L^p(\Omega)}.
\end{equation}

\begin{definition}
For every measurable function $f:\Omega\to\mathbb{R}$ the function $f^*:(0,|\Omega|)\to
\mathbb{R}$ defined by
$$f^*(s)=\sup\{t\in\mathbb{R}: d_f(t)>s\}$$
is called \emph{decreasing rearrangement of $f$}.
\end{definition}
An equivalent definition (used by some authors) is $f^*(s)=\inf\{t\in\mathbb{R}: d_f(t)\leq
s\}$. 

\begin{proposition} \label{furbi}
For each $f$ its decreasing rearrangement $f^*$ is decreasing, right continuous and
we have 
$$\lim_{s\to 0}f^*(s)=\esssup f\quad\text{and}\quad\lim_{s\to|\Omega|}f^*(s)=\essinf f.$$
Moreover, if $F:\mathbb{R}\to\mathbb{R}$ is a Borel measurable function then
$F\circ f\in L^1(\Omega)$ implies $F\circ f^*\in L^1(0,|\Omega|)$ and 
$$\int_\Omega F\circ f\,dx=\int_0^{|\Omega|} F\circ f^*\,ds.$$ 
Finally, $d_{f^*}=d_f$ and, for each measurable function $g$ we have $f\sim g$ if and only if $f^*=g^*$.
\end{proposition}

Some of the previous claims are simple consequences of the definition of $f^*$, for
more details see \cite[Chapter 2]{day70}. 

As before, it follows that, for each $1\leq p\leq\infty$, if $f\in L^p(\Omega)$ then
$f^*\in L^p(0,|\Omega|)$ and $\|f\|_{L^p(\Omega)}=\|f^*\|_{L^p(0,|\Omega|)}.$
 
\begin{definition}\label{prec1}
Given two functions $f,g\in L^1(\Omega)$, we write $g\prec f$ if
$$\int_0^t g^*\,ds\leq \int_0^t f^*\,ds\quad\forall\; 0\leq t\leq|\Omega|\quad\quad
{and}\quad\quad\int_0^{|\Omega|} g^*\,ds= \int_0^{|\Omega|} f^*\,ds.$$ 
\end{definition}
Note that $g\sim f$ if and only if $g\prec f$ and $f\prec g$. Among many properties of the relation
$\prec$ we mention the following (a proof is in \cite[Lemma 8.2]{day70}).

\begin{proposition}\label{prec}
For any pair of functions $f,g\in L^1(\Omega)$ and real numbers $\alpha$ and $\beta$,
if $\alpha\leq f\leq\beta$ a.e. in $\Omega$ and $g\prec f$ then $\alpha\leq g\leq\beta$ a.e. in $\Omega$.
\end{proposition}

\begin{proposition}\label{prec2}
For $f\in L^1(\Omega)$ let $g=\frac{1}{|\Omega|}\int_\Omega f \, dx$. Then we have 
$g\prec f$.
\end{proposition}

\begin{definition}\label{class} 
Let $f:\Omega\to\mathbb{R}$ a measurable function. We call the set
$$\mathcal{G}(f)=\{g:\Omega\to\mathbb{R}: g \text{ is measurable and } g\sim f \}$$
\emph{class of rearrangement of $f$} or \emph{set of rearrangements of $f$}.
\end{definition}

Note that, for $1\leq p\leq\infty$, if $f$ is in $L^p(\Omega)$ then $\mathcal{G}(f)$ is
contained in $L^p(\Omega)$.

As we will see in the next section, we are interested in the optimization of a functional 
defined on a class of rearrangement $\mathcal{G}(\rho_0)$,
where $\rho_0$ belongs to $L^\infty(\Omega)$. For this reason, although almost all of what follows
holds in a much more general context, hereafter we restrict our attention to rearrangement classes of
functions in $L^\infty(\Omega)$.
We need compactness properties of the set $\mathcal{G}(\rho_0)$, with a little
effort it can be showed that this set is closed but in general it is not compact in the norm topology
of $L^\infty(\Omega)$. Therefore we focus our attention on the weak* compactness.
By $\overline{\mathcal{G}(\rho_0)}$ we denote the closure of $\mathcal{G}(\rho_0)$ in the weak* topology
of $L^\infty(\Omega)$.

\begin{proposition}\label{cane}
 Let $\rho_0$ be a function of $L^\infty(\Omega)$. Then $\overline{\mathcal{G}(\rho_0)}$ is
 
 i)  weakly* compact;
 
 ii) metrizable in the weak* topology;
 
 iii) sequentially weakly* compact. 
\end{proposition}
\begin{proof}
i) By \eqref{equi} it follows that $\mathcal{G}(\rho_0)$ is contained in $B_{\|\rho_0\|_
{L^\infty(\Omega)}}=\{f\in L^\infty(\Omega): \|f\|_{L^\infty(\Omega)}\leq\|\rho_0\|_{L^\infty(\Omega)}\}$.
$B_{\|\rho_0\|_{L^\infty(\Omega)}}$ is weakly* compact and then it is also weakly* closed because the weak* topology is Hausdorff. Hence 
$\overline{\mathcal{G}(\rho_0)}$ is a weakly* closed subset of $B_{\|\rho_0\|_{L^\infty(\Omega)}}$ and
thus it is weakly* compact as well.
ii) Owing to the separability of $L^1(\Omega)$, $B_{\|\rho_0\|_{L^\infty(\Omega)}}$ is metrizable in the weak* topology and
the claim follows.
iii) It is an immediate consequence of i) and ii).
\end{proof}

Moreover, the sets $\mathcal{G}(\rho_0)$ and $\overline{\mathcal{G}(\rho_0)}$ have further
properties. 

\begin{definition}
Let $C$ be a convex set of a real vector space. An element $v$ in $C$ is said an \emph{extreme point of
$C$} if for every $u$ and $w$ in $C$ the identity $v=\frac{u}{2}+\frac{w}{2}$ implies $u=w$.
\end{definition}
A vertex of a polygon is an example of extreme point.

\begin{proposition}\label{convexity}
Let $\rho_0$ be a function of $L^\infty(\Omega)$, then

i) $\overline{\mathcal{G}(\rho_0)}=\{f\in L^\infty(\Omega): f\prec \rho_0\}$, 

ii) $\overline{\mathcal{G}(\rho_0)}$ is convex,

iii) $\mathcal{G}(\rho_0)$ is the set of the extreme points of $\overline{\mathcal{G}(\rho_0)}$.
\end{proposition}
\begin{proof}
The claims follow from \cite[Theorems 22.13, 22.2, 17.4, 20.3]{day70}.  
\end{proof}

An evident consequence of the previous theorem is that $\overline{\mathcal{G}(\rho_0)}$ is
the weakly* closed convex hull of $\mathcal{G}(\rho_0)$.

The following is \cite[Theorem 11.1]{day70} rephrased for our case.

\begin{proposition}
Let $u\in L^1(\Omega)$ and $\rho_0\in L^\infty(\Omega)$. Then 
\begin{equation}\label{day}
\int_0^{|\Omega|} \rho_0^*(|\Omega|-s) u^*(s)\, ds\leq\int_\Omega\rho\,u\, dx
\leq\int_0^{|\Omega|} \rho_0^*(s) u^*(s)\, ds\quad\quad\forall\rho\in\mathcal{G}(\rho_0),
\end{equation}
moreover both sides of \eqref{day} are taken on.
\end{proposition}

The previous proposition implies that the linear optimization problems
\begin{equation}\label{max}
 \sup_{\rho \in \mathcal{G}(\rho_0)} \int_\Omega \rho u \, dx\end{equation}
 and
 \begin{equation*}
 \inf_{\rho \in \mathcal{G}(\rho_0)} \int_\Omega \rho u \, dx\end{equation*}
 admit solution.

Finally, we recall the following result proved in \cite[Theorem 5]{B}.

\begin{proposition}\label{Teobart}
Let $u\in L^1(\Omega)$ and $\rho_0\in L^\infty(\Omega)$. If problem \eqref{max} has a unique solution 
$\rho_M$, 
then there exists an increasing function $\psi$ such that $\rho_M=\psi \circ u$ a.e. in 
$\Omega$.
\end{proposition}

\section{Existence of minimizers}\label{main}

Let $\rho_0\in L^\infty(\Omega)$, $\mathcal{G}(\rho_0)$
be the class of rearrangements of $\rho_0$ and $\lambda_k(\rho) $, $\rho \in
\mathcal{G}(\rho_0)$, be the $k$-th positive
eigenvalue of problem \eqref{p0}. In this section we investigate 
the optimization problem 
\begin{equation*}
\inf_{\rho \in \mathcal{G}(\rho_0)} \lambda_1(\rho),
\end{equation*}
which can be expressed in terms of the eigenvalue $\mu_{1}(\rho)$ of the operator $G_\rho$, defined in
\eqref{G}, as
\begin{equation*}
\sup_{\rho \in \mathcal{G}(\rho_0)} \mu_1(\rho).
\end{equation*}
Observe that, by Proposition \ref{segnorho}, 
$\mu_k(\rho)$ and $u_\rho$ (the positive first  eigenfuction of problem \eqref{p0}
normalized as in \eqref{normaliz1}) are well defined only when $|\{\rho>0\}|>0$. We extend them
to the whole space $L^\infty(\Omega)$ by putting
\begin{equation} \label{muk}
\widetilde{\mu}_k(\rho)=\begin{cases} \mu_k(\rho) \quad & \text{if } |\{\rho>0\}|>0\\
0 & \text{if } |\{\rho>0\}|=0\end{cases}
\end{equation}
and
\begin{equation}  \label{urho}
\widetilde{u}_\rho=\begin{cases}u_\rho \quad & \text{if } |\{\rho>0\}|>0\\
0 & \text{if } |\{\rho>0\}|=0.\end{cases}
\end{equation}

\begin{rem}\label{oss1}
Note that $\widetilde{\mu}_k(\rho)=0$ if and only if $\rho\leq 0$ a.e. in $\Omega$ and, in this circumstance,
the inequality
\begin{equation}\label{tilde}\sup_{F_k}
\min_{f\in F_k\atop f\neq 0}
\cfrac{\langle G_\rho f, f \rangle_{H^s_0(\Omega)} }{\|f\|^2_{H_0^{s}(\Omega)}} \,\leq 0\end{equation}
holds, where $F_k$ varies among all the $k$-dimensional subspaces of $H_0^s(\Omega)$.\\
Moreover,  from \eqref{homo}, we have $\widetilde{\mu}_1(\alpha\rho)=\alpha \widetilde{\mu}_1(\rho)$
for every $\alpha\geq 0$.
\end{rem}

\begin{theorem}\label{teo1}
Let $\rho\in L^\infty(\Omega)$,  $G_\rho$ be the linear operator \eqref{G}, $\widetilde{\mu}_k(\rho)$ 
as defined in \eqref{muk} for $k=1,2,3,\ldots$ and $\widetilde{u}_\rho$ as in \eqref{urho}. Then\\
i) the map $\rho\mapsto G_\rho$ is sequentially weakly* continuous from $L^\infty(\Omega)$
to $\mathcal{L}(H_0^s(\Omega),H_0^s(\Omega)) $ endowed with the norm topology;\\
ii) the map $\rho\mapsto \widetilde{\mu}_k(\rho)$ is sequentially weakly* continuous in
$L^\infty(\Omega)$; \\
iii) the map $\rho\mapsto \widetilde{\mu}_1(\rho)\widetilde{u}_\rho$ is sequentially weakly* continuous
from $L^\infty(\Omega)$ to $H_0^s(\Omega)$ (endowed with the norm topology). In particular, 
for any sequence $\{\rho_n\}$ weakly* convergent to $\eta\in L^\infty(\Omega)$, with $\widetilde{\mu}_1(\eta)>0$, then 
$\{\widetilde{u}_{\rho_n}\}$ converges to $\widetilde{u}_{\eta}$ in $H_0^s(\Omega)$.
\end{theorem}

\begin{proof}
i) Let $\{\rho_n\}$ be a sequence which weakly* converges to $\rho$ in $L^\infty(\Omega)$. 
Being $\{\rho_n\}$ bounded in $L^\infty(\Omega)$, there exists a constant $M>0$ such that
\begin{equation}\label{M}|\rho|\leq M\quad \text{ and } \quad |\rho_n|\leq M
\quad \forall\, n.\end{equation}
We begin by proving that 
$G_{\rho_n}(f)$ tends to $G_\rho(f)$ in $H^s_0(\Omega)$ for any fixed $f\in H^s_0(\Omega)$. 
Note that the sequence $\{\rho_n f\}$ is weakly convergent to $\rho f$ in $L^2(\Omega)$, then,
exploiting the compactness of the embedding  $L^2(\Omega)\hookrightarrow H^{-s}(\Omega)$, we conclude that this convergence is also strong in $H^{-s}(\Omega)$. Then
\begin{align*}
\|G_{\rho_n}(f) - G_\rho (f)\|_{H_0^s(\Omega)}& =\|G(\rho_n f-\rho f)\|_{H_0^s(\Omega)} 
\\&\leq \|G\|_{\mathcal{L}
(H^{-s}(\Omega), H_0^{s}(\Omega))} \|\rho_n f-\rho f\|_{H^{-s}(\Omega)}=  \|\rho_n f-\rho f\|_{H^{-s}(\Omega)},
\end{align*}
where we used $G_\rho(f)= G(\rho f)$, with $G$ defined by \eqref{operatore},
 and \eqref{normaduG}.
Therefore $G_{\rho_n}(f)$ converges to $ G_\rho(f)$ in $H_0^s(\Omega)$.

Now, for fixed $n$, let $\{f_{n,k}\}$, $k=1,2, 3,\ldots$, be a maximizing sequence of
$$\sup_{g\in H_0^s(\Omega)\atop \|g\|_{H_0^s(\Omega)}\leq 1}
 \|G_{\rho_n}(g) - G_\rho(g)\|_{H_0^s(\Omega)}=\|G_{\rho_n}-G_\rho\|_{\mathcal{L}(H_0^s(\Omega), H_0^s(\Omega))} .$$
Then, being $\|f_{n,k}\|_{H_0^s(\Omega)}\leq 1$, we can extract a subsequence (still denoted by $\{f_{n,k}\}$)
weakly convergent to some $f_n\in H_0^s(\Omega)$. Since 
$G_{\rho_n}$ and $G_\rho$ are compact operators (see Proposition \ref{self}), it follows that
$G_{\rho_n}(f_{n,k})$ converges to $G_{\rho_n}(f_{n})$ and  
$G_{\rho}(f_{n,k})$ converges to $G_{\rho}(f_{n})$ strongly in $H_0^s(\Omega)$ as $k$ goes to
$\infty$. Thus we find
$$ \|G_{\rho_n}-G_\rho\|_{\mathcal{L}(H_0^s(\Omega), H_0^s(\Omega))}
=\|G_{\rho_n}(f_n) - G_\rho(f_n)\|_{H_0^s(\Omega)}.$$
This procedure yields a sequence $\{f_n\}$ in $H_0^s(\Omega)$ such that $\|f_{n}\|_{H_0^s(\Omega)}\leq 1$ for all $n$. Then, up to a subsequence, we can assume that
$\{f_n\}$ weakly converges  to a function $f\in H_0^s(\Omega)$ 
and (by compactness of the embedding $H_0^{s}(\Omega)\hookrightarrow L^2(\Omega)$) strongly in $L^2(\Omega)$. 
By using \eqref{j}, \eqref{normaduG} and \eqref{M} we find
\begin{align*} & \|G_{\rho_n}- G_\rho\|_{\mathcal{L}(H_0^s(\Omega), H_0^s(\Omega))}=
\|G_{\rho_n}(f_n) - G_\rho(f_n)\|_{H_0^s(\Omega)}\\
& \leq \|G_{\rho_n}(f) - G_\rho(f)\|_{H_0^s(\Omega)} + \|G_{\rho_n}(f_n-f) - G_\rho(f_n-f)\|_{H_0^s(\Omega)}\\
&=\|G_{\rho_n}(f) - G_\rho(f)\|_{H_0^s(\Omega)}+ \|G(\rho_n(f_n-f)-\rho(f_n-f))\|_{H_0^s(\Omega)}\\
&\leq \|G_{\rho_n}(f) - G_\rho(f)\|_{H_0^s(\Omega)}+ \|G\|_{\mathcal{L}(H^{-s}(\Omega), H_0^s(\Omega))}
\left(\|\rho_n(f_n-f)\|_{H^{-s}(\Omega)} +\|\rho(f_n-f))\|_{H^{-s}(\Omega)} \right)\\
&\leq \|G_{\rho_n}(f) - G_\rho(f)\|_{H_0^s(\Omega)} +
 2C M \|f_n-f\|_{L^2(\Omega)}. \end{align*}
Therefore $G_{\rho_n}$ converges to $G_\rho$ in the operator norm. \\

ii) If we show that, for any $k=1, 2, 3, \ldots$ and $\rho, \eta\in L^\infty(\Omega)$ the estimates
\begin{equation}\label{modulo} |\widetilde{\mu}_k(\rho)- \widetilde{\mu}_k(\eta)|\leq  \|G_{\rho}- G_\eta\|_{\mathcal{L}(H_0^s(\Omega), H_0^s(\Omega))}
\end{equation}
hold, then the claim follows immediately from i).\\
We split the argument in three cases.

Case 1.  $\widetilde{\mu}_k(\rho), \ \widetilde{\mu}_k(\eta)>0$.\\ Following \cite[Theorem 2.3.1]{He} and by means of the Fischer's Principle \eqref{Fischer1} we have
\begin{equation*} \begin{split}
\widetilde{\mu}_k(\rho)- \widetilde{\mu}_k(\eta) &=  \max_{F_k}
\min_{f\in F_k\atop f\neq 0}
\cfrac{\langle G_\rho f, f \rangle_{H^s_0(\Omega)} }{\|f\|^2_{H_0^{s}(\Omega)}} \,
- \max_{F_k}
\min_{f\in F_k\atop f\neq 0}
\cfrac{\langle G_\eta f, f \rangle_{H^s_0(\Omega)} }{\|f\|^2_{H_0^{s}(\Omega)}} \,\\
& \leq 
\min_{f\in F_k(\rho)\atop f\neq 0}
\cfrac{\langle G_\rho f, f \rangle_{H^s_0(\Omega)} }{\|f\|^2_{H_0^{s}(\Omega)}} \,
- 
\min_{f\in F_k(\rho)\atop f\neq 0}
\cfrac{\langle G_\eta f, f \rangle_{H^s_0(\Omega)} }{\|f\|^2_{H_0^{s}(\Omega)}} \,\\
& \leq 
\cfrac{\langle G_\rho f_\eta, f_\eta \rangle_{H^s_0(\Omega)} }{\|f_\eta\|^2_{H_0^{s}(\Omega)}} \,-
\cfrac{\langle G_\eta f_\eta, f_\eta \rangle_{H^s_0(\Omega)} }{\|f_\eta\|^2_{H_0^{s}(\Omega)}} \,\\
& = 
\cfrac{\langle (G_\rho-G_\eta)f_\eta, f_\eta \rangle_{H^s_0(\Omega)} }
{\|f_\eta\|^2_{H_0^{s}(\Omega)}} \,\leq  
 \|G_{\rho}- G_\eta\|_{\mathcal{L}(H_0^s(\Omega), H_0^s(\Omega))},
\end{split} 
\end{equation*}
where $F_k(\rho)$ is a $k$-dimensional subspace of $H^s_0(\Omega)$ such that
$$\max_{F_k}
\min_{f\in F_k\atop f\neq 0}
\cfrac{\langle G_\rho f, f \rangle_{H^s_0(\Omega)} }{\|f\|^2_{H_0^{s}(\Omega)}} \, =\min_{f\in F_k(\rho)\atop f\neq 0}
\cfrac{\langle G_\rho f, f \rangle_{H^s_0(\Omega)} }{\|f\|^2_{H_0^{s}(\Omega)}} \,$$
and $f_\eta$ is a function in $F_k(\rho)$ such that
$$ \min_{f\in F_k(\rho)\atop f\neq 0}
\cfrac{\langle G_\eta f, f \rangle_{H^s_0(\Omega)} }{\|f\|^2_{H_0^{s}(\Omega)}} \,=
\cfrac{\langle G_\eta f_\eta, f_\eta \rangle_{H^s_0(\Omega)} }
{\|f_\eta\|^2_{H_0^{s}(\Omega)}} \,.
$$
Interchanging the role of $\rho$ and $\eta$ we find \eqref{modulo}.\\

Case 2. $\widetilde{\mu}_k(\rho)>0$, $ \widetilde{\mu}_k(\eta)=0$ (and similarly in the case $\widetilde{\mu}_k(\eta)>0$, $ \widetilde{\mu}_k(\rho)=0$).\\
Note that in this case \eqref{tilde} holds for the weight function $\eta$. Then
the previous argument still applies provided that we replace the first step of the inequality chain by 
\begin{equation*} |\widetilde{\mu}_k(\rho)- \widetilde{\mu}_k(\eta) |=
\widetilde{\mu}_k(\rho) \leq   \max_{F_k}
\min_{f\in F_k\atop f\neq 0}
\cfrac{\langle G_\rho f, f \rangle_{H^s_0(\Omega)} }{\|f\|^2_{H_0^{s}(\Omega)}} \,
- \sup_{F_k}
\min_{f\in F_k\atop f\neq 0}
\cfrac{\langle G_\eta f, f \rangle_{H^s_0(\Omega)} }{\|f\|^2_{H_0^{s}(\Omega)}} \,.\end{equation*}

Case 3. $\widetilde{\mu}_k(\rho)= \widetilde{\mu}_k(\eta)=0$.\\
In this case \eqref{modulo} is obvious.\\

Therefore statement ii) is proved.\\

iii) Let $\{\rho_n\}, \rho$ be such that $\rho_n$ is weakly$^*$ convergent
to $\rho$ in $L^\infty(\Omega)$. Being 
$\|\widetilde{u}_{\rho_n}\|_{H_0^s(\Omega)}\leq1$,
up to a subsequence we can assume that $\widetilde{u}_{\rho_n}$ is weakly convergent to $z
\in H_0^s(\Omega)$, strongly in $L^2(\Omega)$ and pointwisely a.e. in $\Omega$. \\
First suppose $\widetilde{\mu}_1(\rho)=0$. Then, by ii) $\widetilde{\mu}_1(\rho_n) \widetilde{u}_{\rho_n}$
weakly converges in $H_0^s(\Omega)$ to $\widetilde{\mu}_1(\rho)z=0= \widetilde{\mu}_1(\rho) \widetilde{u}_{\rho}$. Moreover, $\|\widetilde{\mu}_1(\rho_n) \widetilde{u}_{\rho_n}\|_{H_0^s(\Omega)}
=\widetilde{\mu}_1(\rho_n)\| \widetilde{u}_{\rho_n}\|_{H_0^s(\Omega)}$ tends to $0=\|\widetilde{\mu}_1(\rho) \widetilde{u}_{\rho}\|_{H_0^s(\Omega)}$. Therefore $\widetilde{\mu}_1(\rho_n) \widetilde{u}_{\rho_n}$
strongly converges to $\widetilde{\mu}_1(\rho) \widetilde{u}_{\rho}$ in $H_0^s(\Omega)$.\\
Next, consider the case $\widetilde{\mu}_1(\rho)>0$. By ii) we have $\widetilde{\mu}_1(\rho_n)>0$
for all $n$ large enough. This implies  $\widetilde{\mu}_1(\rho_n)
=\frac{1}{\lambda_1(\rho_n)}\, $ and $\widetilde{u}_{\rho_n}= u_{\rho_n}$. Positiveness and pointwise convergence of $u_{\rho_n}$ to $z$ imply $z\geq 0$ a.e. in $\Omega$.
Moreover, by \eqref{normaliz2} we have
$$ \int_\Omega \rho_n u^2_{\rho_n} \, dx =\cfrac{1}{\lambda_1(\rho_n)}\, $$
and by ii), passing to the limit, we find 
$$ \int_\Omega \rho z^2 \, dx =\cfrac{1}{\lambda_1(\rho)}\, ,$$
which implies $z\neq 0$. By using the weak form of problem \eqref{p0} for $u_{\rho_n}$
we have
\begin{equation*}\langle u_{\rho_n}, \varphi\rangle_{H_0^s(\Omega)} = 
\lambda_1(\rho_n) \langle \rho_n u_{\rho_n}, \varphi\rangle_{L^2(\Omega)} = \lambda_1(\rho_n)
\int_\Omega  \rho_n u_{\rho_n} \varphi \, dx \quad \forall\, \varphi \in H_0^s(\Omega)
\end{equation*} and, letting $n$ to infinity, we deduce $z=u_\rho$.\\
By ii) $\mu_1(\rho_n) u_{\rho_n}$
weakly converges in $H_0^s(\Omega)$ to $\mu_1(\rho)u_{\rho}$ and 
 $\|\mu_1(\rho_n)u_{\rho_n}\|_{H_0^s(\Omega)}
=\mu_1(\rho_n)$ tends to $\mu_1(\rho) =\|\mu_1(\rho) u_{\rho}\|_{H_0^s(\Omega)}$.
Hence $\mu_1(\rho_n) u_{\rho_n}$
strongly converges to $\mu_1(\rho)u_{\rho}$ in $H_0^s(\Omega)$.\\
The last claim is immediate provided one observes that $\widetilde{\mu}_1(\eta)>0$ 
implies $\widetilde{\mu}_1(\rho_n)>0$ for all $n$ large enough.
\end{proof}

\begin{theorem}\label{teo2} Let $\rho, \eta, \rho_0\in L^\infty(\Omega)$,  
$\widetilde{\mu}_1(\rho)$ be defined as in \eqref{muk} for $k=1$ and $\overline{\mathcal{G}(\rho_0)}$ the weak* closure in $L^\infty(\Omega)$ of the class of rearrangement $\mathcal{G}(\rho_0)$ introduced in 
Definition \ref{class}. Then\\
i) the map $\rho\mapsto \widetilde{\mu}_1(\rho)$ 
 is convex on $L^\infty(\Omega) $; \\
ii) if $\rho$ and  $\eta$ are linearly indipendent and $ \widetilde{\mu}_1(\rho),  \widetilde{\mu}_1(\eta)>0$, then 
\begin{equation*} \widetilde{\mu}_1(t\rho+(1-t)\mu)< t \widetilde{\mu}_1(\rho)+(1-t)  \widetilde{\mu}_1(\eta)
\end{equation*}
for all $0<t<1$;\\
iii)
if $\int_\Omega \rho_0 \, dx >0$, then the map 
$\rho\mapsto \widetilde{ \mu}_1(\rho)$ 
 is strictly convex on $\overline{\mathcal{G}(\rho_0)}\, $.
\end{theorem}

\begin{proof}

i) The Fischer's Principle \eqref{Fischer1} and \eqref{tilde} both for $k=1$ yield
\begin{equation}\label{basta}
\sup_{f\in H_0^s(\Omega) \atop f\neq 0}
\cfrac{\langle G_\rho f, f \rangle_{H^s_0(\Omega)} }{\|f\|^2_{H_0^{s}(\Omega)}} \,\leq \widetilde{ \mu}_1(\rho)
\end{equation}
for every $\rho \in L^\infty(\Omega)$. Moreover, if $\widetilde{ \mu}_1(\rho)>0$, then equality sign holds and
the supremum is attained when $f$ is an eigenfunction of $\mu_1(\rho)=\widetilde{\mu}_1(\rho)$.
Let $\rho, \mu\in L^\infty(\Omega)$, $0\leq t\leq 1$. We show that
\begin{equation}\label{conv}
\widetilde{\mu}_1(t\rho + (1-t)\eta)\leq t \widetilde{\mu}_1 (\rho) +(1-t) \widetilde{\mu}_1 (\eta).
\end{equation}
If $\widetilde{\mu}_1(t\rho + (1-t)\eta)=0$ \eqref{conv} is obvious. Suppose 
$\widetilde{\mu}_1(t\rho + (1-t)\eta)>0$. Then, for all $f\in H_0^s(\Omega)$ we have
\begin{equation} \label{A1} \cfrac{\langle G_{t\rho+(1-t)\eta} f, f \rangle_{H^s_0(\Omega)} }{\|f\|^2_{H_0^{s}(\Omega)}} 
\,=t\,  \cfrac{\langle G_{\rho} f, f \rangle_{H^s_0(\Omega)} }{\|f\|^2_{H_0^{s}(\Omega)}}\, 
+(1-t)\,  \cfrac{\langle G_{\eta} f, f \rangle_{H^s_0(\Omega)} }{\|f\|^2_{H_0^{s}(\Omega)}} \,
\leq t \widetilde{\mu}_1 (\rho) +(1-t) \widetilde{\mu}_1 (\eta),\end{equation}
where we used \eqref{linearit} and \eqref{basta}. Taking the supremum in the left-hand term
and using \eqref{basta} again with equality sign
we find \eqref{conv}.\\

ii) Arguing by contradiction, we suppose that 
equality holds in \eqref{conv}. We find out that $\rho$ and $\mu$ are linearly dependent. 
Equality sign in \eqref{conv} implies 
 $\widetilde{\mu}_1(t\rho + (1-t)\eta)>0$, then (by \eqref{basta}) the equality also holds in \eqref{A1} with 
$f=u=u_{t\rho+(1-t)\eta }$.  We get
$$ \cfrac{\langle G_{\rho} u , u \rangle_{H^s_0(\Omega)}} {\|u\|^2_{H_0^{s}(\Omega)}}\, =\widetilde{\mu}_1(\rho) \,\quad
 \text{ and } \quad \cfrac{\langle G_{\eta} u, u
 \rangle_{H^s_0(\Omega)} }{\|u\|^2_{H_0^{s}(\Omega)}}\, 
 =\widetilde{\mu}_1(\eta).$$
 The simplicity 
of the first eigenvalue, the positiveness of $u$ and the normalization \eqref{normaliz1} imply that 
$u=u_\rho=u_\eta$.
Writing the problem \eqref{p0} in weak form for both weigths $\rho$ and $\eta$ we have
$$ \langle u, \varphi \rangle_{H_0^{s}(\Omega)} =\cfrac{1}{ \widetilde{\mu}_1(\rho)}\, 
\langle \rho u, \varphi\rangle_{L^2(\Omega)} \quad \forall\, \varphi \in H_0^{s}(\Omega) 
$$
and
$$\langle u, \varphi\rangle_{H_0^{s}(\Omega)} =  \cfrac{1}{ \widetilde{\mu}_1(\eta)}\, 
\langle \eta u, \varphi\rangle_{L^2(\Omega)} \quad \forall\, \varphi\in H_0^{s}(\Omega) .$$
Taking the difference of these identities we find
$$\left\langle \left( \cfrac{\rho}{ \widetilde{\mu}_1(\rho)}\, -\cfrac{\eta}{ \widetilde{\mu}_1(\eta)}\, \right) u, \varphi\right\rangle_{L^2(\Omega)} =0 \quad \forall\, \varphi\in H_0^{s}(\Omega), $$
which gives $\rho\widetilde{\mu}_1(\eta)-\eta\widetilde{\mu}_1(\rho)=0$, i.e. 
 $\rho$ and $\eta$ are linearly dependent.\\

iii) First, note that $\int_\Omega \rho\, dx=\int_\Omega \rho_0\, dx>0$ 
 for any $\rho \in \overline{\mathcal{G}(\rho_0)}$. This follows easily by i) of Proposition \ref{convexity}, Definition \ref{prec1} and Proposition \ref{furbi}. Therefore, we have $|\{\rho >0\}|>0$ and thus
 $\widetilde{\mu}_1 (\rho)>0$
for all $\rho \in \overline{\mathcal{G}(\rho_0)}$. Next, we show that any distinct functions $\rho$ and $\eta$ in
$\overline{\mathcal{G}(\rho_0)}$ are linearly independent. Indeed, let $\alpha \rho +\beta \eta =0$
with $\alpha, \beta\in \mathbb{R}$. Integrating over $\Omega$ we obtain
$(\alpha+\beta)\int_\Omega \rho_0\, dx=0$, which implies $\beta=-\alpha$ and, in turn, $\alpha(\rho - \eta)=0$
and $\alpha =0$. Hence, $\rho$ and $\eta$ are linearly independent. The statement
is now an immediate consequence of ii)
\end{proof}

\begin{rem} If $\int_\Omega \rho_0\, dx=0$, $\rho_0\neq 0$, the map $\rho\mapsto \widetilde{\mu}_1(\rho)$ is not
strictly convex on $\overline{\mathcal{G}(\rho_0)}$. In fact, in this case (by Proposition \ref{prec2}) the null function 
belongs to $\overline{\mathcal{G}(\rho_0)}$. By convexity of $\overline{\mathcal{G}(\rho_0)}$ 
(see Proposition \ref{convexity}), $\alpha\rho_0\in \overline{\mathcal{G}(\rho_0)}$ for every $\alpha\in [0,1]$ and, by Remark \ref{oss1}, we have 
$\widetilde{\mu}_1(\alpha\rho_0)=\alpha \widetilde{\mu}_1(\rho_0)$, which excludes strict convexity.
\end{rem}

For the definitions and some basic results on the G\^ateaux differentiability
we refer the reader to \cite{ET}.

\begin{theorem}\label{teo3} 
Let $\rho\in L^\infty(\Omega)$,  $\widetilde{\mu}_1(\rho)$ be
 defined as in \eqref{muk} for $k=1$ and $u_\rho$ denote the first positive 
eigenfunction of problem \eqref{p0} normalized as in \eqref{normaliz1}.
The map $\rho\mapsto  \widetilde{\mu}_1(\rho)$ 
is G\^ateaux differentiable at any $\rho$ such that $ \widetilde{\mu}_1(\rho)>0$
with G\^ateaux differential equal to $u_\rho^2$. In other words,
 for every direction $v\in L^\infty(\Omega)$
  we have 
\begin{equation}\label{gatto} \widetilde{\mu}_1'(\rho; v) =\int_\Omega u_\rho^2 v\, dx. \end{equation}
\end{theorem}

\begin{proof}
Let us compute 
$$\lim_{t\to 0} \cfrac{ \widetilde{\mu}_1(\rho+ t v)- 
\widetilde{\mu}_1(\rho)}{t}\, .$$
Note that, by ii) of Theorem \ref{teo1}, $\widetilde{\mu}_1(\rho+ t v)$ converges to
$\widetilde{\mu}_1(\rho)$ as $t$ goes to zero for any $\rho, v\in L^\infty(\Omega)$. Therefore, 
$\widetilde{\mu}_1(\rho+ t v)>0$ for $t$ small enough.

The eigenfunctions $u_{\rho}$ and $u_{\rho+tv}$ satisfy
$$\widetilde{\mu}_1(\rho)\langle u_\rho, \varphi\rangle_{H_0^{s}(\Omega)} =
\langle \rho u_\rho, \varphi\rangle_{L^2(\Omega)} \quad \forall\, \varphi\in H_0^{s}(\Omega) $$
and
$$\widetilde{\mu}_1(\rho+ t v)\langle u_{\rho+tv}, \varphi\rangle_{H_0^{s}(\Omega)} =
\langle (\rho+tv) u_{\rho+tv}, \varphi\rangle_{L^2(\Omega)} \quad \forall\, \varphi\in 
H_0^{s}(\Omega). $$
By choosing $\varphi=u_{\rho+tv}$ in the former equation, $\varphi=u_\rho$ in the latter
and comparing we get
\begin{equation*}
\widetilde{\mu}_1(\rho+ t v)
\langle \rho u_\rho, u_{\rho+tv}\rangle_{L^2(\Omega)}=
\widetilde{\mu}_1(\rho) 
\langle (\rho+tv) u_{\rho+tv}, u_\rho\rangle_{L^2(\Omega)}.
\end{equation*}
Rearranging we find
\begin{equation}\label{rap} \cfrac{\widetilde{\mu}_1(\rho+ t v)- 
\widetilde{\mu}_1(\rho)}{t}\, \int_\Omega \rho \, u_\rho
u_{\rho +tv}\, dx =\widetilde{\mu}_1(\rho) \int_\Omega  u_\rho
u_{\rho +tv} v\, dx .\end{equation}
If $t$ goes to zero,  
then by iii) of Theorem \ref{teo1} it follows that  
$u_{\rho+ t v}$ converges to $u_\rho$ in $H_0^{s}(\Omega)$ and therefore
in $L^2(\Omega)$. Passing to the limit 
in \eqref{rap} and using \eqref{normaliz2} we conclude
\begin{equation*}
\lim_{t\to 0} \cfrac{\widetilde{\mu}_1(\rho+ t v)- \widetilde{\mu}_1(\rho)}{t}\, =
\int_\Omega u_\rho^2 v\, dx,
\end{equation*}
i.e. \eqref{gatto} holds.\end{proof}

We are now able to prove our main result.

\begin{theorem}\label{exist}
Let $\lambda_1(\rho)$ be the first positive eigenvalue of problem
\eqref{p0}, 
$\rho_0\in L^\infty(\Omega)$ 
such that $|\{\rho_0>0\}|>0$ and $\mathcal{G}(\rho_0)$ the class of rearrangement
of $\rho_0$ introduced in Definition \ref{class}. 
Then\\
i) there exist $\check{\rho}_1\in\mathcal{G}(\rho_0)$
such that
\begin{equation}\label{inf}
\lambda_1(\check{\rho}_1)=\min_{\rho\in \mathcal{G}(\rho_0)} \lambda_1(\rho);
\end{equation}
ii) there exists an increasing function $\psi$ such
that 
\begin{equation}\label{carat}\check{\rho}_1= \psi(u_{\check{\rho}_1}) \quad \text{a.e. in }\Omega, \end{equation} where $u_{\check{\rho}_1}$ is 
the positive first eigenfunction relative to $\lambda_1(\check{\rho}_1)$ normalized as in \eqref{normaliz1}.
\end{theorem}

\begin{proof}
i) By iii) of Proposition \ref{cane} and ii) of Theorem \ref{teo1}, $\overline{\mathcal{G}(\rho_0)}$ 
is sequentially weakly* continuous and the map $\rho \mapsto \widetilde{\mu}_1(\rho)$
is sequentially weakly* compact. Therefore, there exists $\check{\rho}_1\in 
\overline{\mathcal{G}(\rho_0)}$ such that 
\begin{equation*}
\widetilde{\mu}_1(\check{\rho}_1)=\max_{\rho\in \overline{\mathcal{G}(\rho_0)}}\widetilde{ \mu}_1(\rho) .
\end{equation*}
Note that, by Proposition \ref{segnorho}, 
the condition $|\{\rho_0>0\}|>0$ guarantees $\widetilde{\mu}_1(\check{\rho}_1)>0$. 

Let us show
 that $\check{\rho}_1$ actually belongs to $\mathcal{G}(\rho_0)$ (in fact,
 the following argument shows that there are not maximizers of $\widetilde{\mu}_1(\rho)$
 in $\overline{\mathcal{G}(\rho_0)}\smallsetminus \mathcal{G}(\rho_0)$). 
 Proceeding by contradiction, 
 suppose that $\check{\rho}_1\not \in\mathcal{G}(\rho_0) $. Then, by iii) of Proposition
 \ref{convexity}, $\check{\rho}_1$ is not an extreme point of $\overline{\mathcal{G}(\rho_0)}$
 and thus there exist $\rho, \eta \in \overline{\mathcal{G}(\rho_0)}$ such that $\rho\neq \eta$
 and $\check{\rho}_1= \frac{\rho + \eta}{2}\, $.
 By i) of Theorem \ref{teo2} and being $\check{\rho}_1$ a maximizer, we have 
$$\widetilde{\mu}_1(\check{\rho}_1)\leq
 \cfrac{\widetilde{\mu}_1(\rho) +\widetilde{\mu}_1(\eta)}{2}\,
\leq \widetilde{\mu}_1(\check{\rho}_1)$$ and
then, equality sign holds. This implies 
$\widetilde{\mu}_1(\rho) =\widetilde{\mu}_1(\eta)=\widetilde{\mu}_1(\check{\rho}_1)>0$, 
that is $\rho$ and $\eta$ are maximizers as well. Now, applying ii) of Theorem \ref{teo2}
to $\rho$ and $\eta$ with $t=\frac{1}{2}\, $, we conclude that $\rho$ and $\eta$ are linearly 
dependent. Without loss of generality, we can assume that there exists $\alpha \in \mathbb{R}$ such 
that $\eta =\alpha \rho$, moreover $\alpha$ is nonzero since $\eta$ is a maximizer. Combining
 $\eta =\alpha \rho$ with $\check{\rho}_1= \frac{\rho + \eta}{2}\, $ we get 
 $\check{\rho}_1= \frac{1+\alpha}{2}\, \rho=\frac{1+\alpha}{2\alpha}\, \eta$. It is immediate
 to show that at least one of the coefficients $\frac{1+\alpha}{2}\,$ and $\frac{1+\alpha}{2\alpha}\,$
 must be nonnegative. In either cases we find a contradiction. For instance, if $\frac{1+\alpha}{2\alpha}\,\geq 0$, by Remark \ref{oss1} and maximality of $\eta$ we obtain 
 $$\widetilde{\mu}(\check{\rho}_1)= \cfrac{1+\alpha}{2\alpha}\,\widetilde{\mu} (\eta)=
 \cfrac{1+\alpha}{2\alpha}\,\widetilde{\mu} (\check{\rho}_1),$$
 which implies $\alpha=1$ and yields the contradiction $\rho=\eta$. The other case is analogous.
Thus, we conclude that $\check{\rho}_1\in \mathcal{G}(\rho_0)$ and 
\begin{equation}\label{topo}
\widetilde{\mu}_1(\check{\rho}_1)=\max_{\rho\in \mathcal{G}(\rho_0)}\widetilde{ \mu}_1(\rho) .
\end{equation}
Being $|\{\rho_0 >0\}|>0$, we have 
$$\lambda_1(\rho)= \cfrac{1}{\mu_1(\rho)}\, =\cfrac{1}{\widetilde{\mu}_1(\rho)}\, $$
for all $\rho\in\mathcal{G}(\rho_0)$. Therefore, \eqref{topo} is equivalent to 
\eqref{inf} and i) is proved.\\

ii) We prove the claim by using Proposition \ref{Teobart}; more precisely, we show that 
\begin{equation}\label{silvia}
\int_\Omega 
\check{\rho}_1 u_{\check{\rho}_1}^2\, dx> \int_\Omega 
\rho\, u_{\check{\rho}_1}^2\, dx\end{equation}
for every $\rho\in \overline{\mathcal{G}(\rho_0)}\smallsetminus \{\check{\rho}_1\}$.
By exploiting the convexity of $\widetilde{\mu}_1(\rho)$ (see Theorem \ref{teo2}) and its 
G\^ateaux differentiability in $\check{\rho}_1$ (see Theorem \ref{teo3})
 we have (for details see \cite{ET})
 \begin{equation} \label{maria}
\widetilde{\mu_1}(\rho)\geq 
\widetilde{\mu_1}\big(\check{\rho}_1)
+\int_\Omega (\rho-
\check{\rho}_1) u_{\check{\rho}_1}^2\, dx
\end{equation} 
for all $\rho\in \overline{\mathcal{G}(\rho_0)}$.
First, let us suppose $\widetilde{\mu_1}(\rho)< \widetilde{\mu_1}(\check{\rho}_1)$.
Comparing with \eqref{maria} we find
\begin{equation*}
\int_\Omega 
(\rho-\check{\rho}_1) u_{\check{\rho}_1}^2\, dx<0  ,\end{equation*}
that is \eqref{silvia}.\\
Next, let us consider the case $\widetilde{\mu_1}(\rho)=\widetilde{\mu_1}(\check{\rho}_1)$,
$\rho\in \overline{\mathcal{G}(\rho_0)}\smallsetminus \{\check{\rho}_1\}$. By the argument 
used in part i) there are not maximizers of $\widetilde{\mu_1}$ in $\overline{\mathcal{G}(\rho_0)}\smallsetminus \mathcal{G}(\rho_0)$, therefore $\rho \in \mathcal{G}(\rho_0)$.

If $\check{\rho}_1$ and $\rho$ are linearly independent, then, ii) of Theorem \ref{teo2}
implies 
$$ \widetilde{\mu}_1 \left(\frac{\check{\rho}_1 + \rho}{2} \right)<\frac{\widetilde{\mu}_1(
\check{\rho}_1) +\widetilde{\mu}_1(\rho)}{2}\, = \widetilde{\mu}_1(
\check{\rho}_1).$$
Then, as in the previous step, \eqref{maria} with $\frac{\check{\rho}_1 + \rho}{2}\,$ in place of 
$\rho$ yields \eqref{silvia}.\\
Finally, let $\check{\rho}_1$ and $\rho$ be linearly dependent. Being $\check{\rho}_1$ and 
$\rho$ both nonzero, we can assume $\rho=\alpha \check{\rho}_1$
for a constant $\alpha\in \mathbb{R}$. Therefore $|\rho|= |\alpha|\,  |\check{\rho}_1|$. Now, 
by i) and ii) of Proposition \ref{rospo}, the functions $|\rho|$ and $| \check{\rho}_1|$
are equimeasurable and $\esssup |\rho|= \esssup |\check{\rho}_1|>0$.
This leads to $|\alpha|=1$ and, being $\rho$ and $\check{\rho}_1$ distinct, $\alpha=-1$. Thus $\rho=
-\check{\rho}_1$, which by \eqref{normaliz2} gives 
$$\int_\Omega 
\rho\,  u_{\check{\rho}_1}^2\, dx=-\int_\Omega \check{\rho}_1 
u_{\check{\rho}_1}^2\, dx = -  \widetilde{\mu}_1(\check{\rho}_1)<  \widetilde{\mu}_1(
\check{\rho}_1) = \int_\Omega \check{\rho}_1
\,  u_{\check{\rho}_1}^2\, dx, $$
i.e. \eqref{silvia}. This completes the proof.
\end{proof}

\begin{rem}
If $\rho_0$ satisfies the stronger condition $\int_\Omega \rho_0\, dx >0$, then the proof
simplifies as one can rely on iii) of Theorem \ref{teo2} (strict convexity of $\rho \mapsto 
\widetilde{\mu}_1(\rho)$). Indeed, from $\check{\rho}_1=\frac{\rho +\eta}{2}\, $, $\rho\neq \eta$, 
it follows immediately the contradiction
$$ \widetilde{\mu}_1(\check{\rho}_1)< \cfrac{\widetilde{\mu}_1(\rho)
 +\widetilde{\mu}_1(\eta)}{2}\, \leq  \widetilde{\mu}_1(\check{\rho}_1).$$
Further, note that in this case $\lambda_1(\rho)$ is well defined for all $\rho\in 
\overline{\mathcal{G}(\rho_0)}$ (it follows by i) of Proposition \ref{convexity}, Definition 
\ref{prec1} and Proposition \ref{furbi} with $F$ equal to the identity function).
However, the previous proof shows that no minimizer of $\lambda_1(\rho)$ belongs to
$\overline{\mathcal{G}(\rho_0)} \smallsetminus \mathcal{G}(\rho_0)$.\\
Finally, in this case the estimate
\begin{equation}\label{stima}\lambda_1(\check{\rho}_1)\leq \cfrac{\lambda_1\, |\Omega|}{
\int_\Omega \rho_0 \, dx}
\end{equation} holds, where
 $\lambda_1$ denotes the first eigenvalue of problem \eqref{p0} with $\rho\equiv 1$.
The estimate \eqref{stima} follows by the fact that the constant function
$c=\frac{1}{|\Omega|} \int_\Omega \rho_0 \, dx$ belongs to $\overline{\mathcal{G}(\rho_0)}$
(see Proposition \ref{prec2}), 
the minimality of $\lambda_1(\check{\rho}_1)$ and the identity
$\lambda_1 =c\lambda_1(c)$ (which is a straightforward consequence of the variational characterization
 \eqref{2a}).
\end{rem}

\begin{rem} The study of the maximization of $\lambda_1(\rho)$ on $\mathcal{G}(\rho_0)$
seems to be rather different. We list here some partial results. Assume $|\{\rho_0>0\}|>0$. If
$\int_\Omega \rho_0 \, dx\leq 0$, then, by Proposition \ref{prec2} and Proposition \ref{convexity},
the nonpositive constant function $c=\frac{1}{|\Omega|}\, \int_\Omega \rho_0 \, dx$ belongs to 
$\overline{\mathcal{G}(\rho_0)}$. Therefore, by definition of $\widetilde{\mu}_1(\rho)$, 
$\min_{\rho \in \overline{\mathcal{G}(\rho_0)}} \widetilde{\mu}_1(\rho)=0$ which, in turns, 
being $\mathcal{G}(\rho_0)$ dense in $\overline{\mathcal{G}(\rho_0)}$ and 
$\widetilde{\mu}_1(\rho)$
sequentially weak* continuous, implies 
$\inf_{\rho \in \mathcal{G}(\rho_0)} \widetilde{\mu}_1(\rho)=0$
and, finally, $\sup_{\rho \in \mathcal{G}(\rho_0)} \lambda_1(\rho)=+\infty$.\\
If, instead $\int_\Omega \rho_0 \, dx>0$, then by proceeding as in the first part of the previous proof 
and using iii) of Theorem \ref{teo2}, one immediately concludes that there is a unique $\widehat{\rho}_1
\in  \overline{\mathcal{G}(\rho_0)}$ such that
$$  \widetilde{\mu}_1(\widehat{\rho}_1)=\min_{\rho \in \overline{\mathcal{G}(\rho_0)}}
 \widetilde{\mu}_1(\rho),$$
 which, in this case, is equivalent to 
$$  \lambda_1(\widehat{\rho}_1)=\max_{\rho \in \overline{\mathcal{G}(\rho_0)}}
 \lambda_1(\rho).$$
 Moreover, by Theorem \ref{teo3}, for all $\rho\in \overline{\mathcal{G}(\rho_0)}$ and
 $t\in (0, 1]$ we can write 
 $$ \widetilde{\mu}_1(\widehat{\rho}_1)\leq  \widetilde{\mu}_1(\widehat{\rho}_1+ 
 t(\rho -\widehat{\rho}_1)) =\widetilde{\mu}_1(\widehat{\rho}_1)+ t \int_\Omega 
u^2_{\widehat{\rho}_1} (\rho -\widehat{\rho}_1)\, dx + o(t) $$
for $t$ that goes to zero. Finally, after some easy algebraic manipulations and passing to the limit we find
$$\int_\Omega \hat{\rho_1} u^2_{\hat{\rho_1}}\, dx \leq \int_\Omega \rho u^2_{\hat{\rho_1}}\, dx
\quad \forall\, \rho \in\overline{\mathcal{G}(\rho_0)} .$$\\
\end{rem}

\begin{rem}\label{ape}
As already note in Remark \ref{osserv}, we have 
$\lambda_{-1}(\rho)= -\lambda_{1}(-\rho)$ for all $\rho \in L^\infty(\Omega)$ such that $|\{\rho<0\}|>0$. Furthermore, it is easy to see from \eqref{rana} that the eigenspaces relative to $\lambda_{-1}(\rho)$ and $\lambda_{1}(-\rho)$ coincide. Finally, observe that
by iii) of Proposition \ref{rospo} with $F(t)=-t$, it follows that $\mathcal{G}(-\rho_0)=-\mathcal{G}(\rho_0)=
\{\rho \in L^\infty (\Omega): -\rho \in \mathcal{G}(\rho_0)\}$ and then $\overline{\mathcal{G}(-\rho_0)}=-\overline{\mathcal{G}(\rho_0)}$. Thus, Theorem \ref{exist} can be reformulated in terms of the first negative eigenvalue $\lambda_{-1}(\rho)$ as follows.\\

\begin{theorem}
Let $\lambda_{-1}(\rho)$ be the first negative eigenvalue of problem
\eqref{p0}, 
$\rho_0\in L^\infty(\Omega)$ 
such that $|\{\rho_0<0\}|>0$ and $\mathcal{G}(\rho_0)$ the class of rearrangement
of $\rho_0$ introduced in Definition \ref{class}. 
Then\\
i) there exist $\check{\rho}_{-1}\in\mathcal{G}(\rho_0)$
such that
\begin{equation}\label{inf2}
\lambda_{-1}(\check{\rho}_{-1})=\max_{\rho\in \mathcal{G}(\rho_0)} \lambda_{-1}(\rho);
\end{equation}
ii) there exists a decreasing function $\phi$ such
that 
\begin{equation*}\check{\rho}_{-1}= \phi(u_{-\check{\rho}_{-1}}) \quad 
\text{a.e. in }\Omega, \end{equation*} where $u_{-\check{\rho}_{-1}}$ is 
the first positive eigenfunction relative to $\lambda_1(-\check{\rho}_{-1})$ normalized as in \eqref{normaliz1}.
\end{theorem}

\end{rem}

\section{Steiner symmetry}\label{symmetry}

We introduce first the definitions and some results about the Steiner  
symmetrization of sets and functions. For a thorough treatment we refer the reader to
\cite{Bro}. Then, we prove our symmetry result.

Let $l(x')= \{x=(x_1,x')\in \mathbb{R}^N: x_1\in \mathbb{R}\}$ for any $x'\in \mathbb{R}^{N-1}$ fixed
and let  $T$ be the hyperplane $\{x=(x_1,x')\in \mathbb{R}^N: x_1=0\}$.

\begin{definition}
Let $\Omega\subset \mathbb{R}^N$ be a measurable set. Then\\
i) the set 
\begin{equation*}
\Omega^\sharp=\left\{x=(x_1,x')\in\mathbb{R}^N:2|x_1|<|\Omega \cap
l(x')|_1, x' \in\mathbb{R}^{N-1}\right\},
\end{equation*}
where $|\,\cdot\, |_1$ denotes the one dimensional Lebesgue measure, is said 
\emph{Steiner symmetrization} of $\Omega$ with respect to the hyperplane $T$;\\
ii) the set $\Omega$ is said \emph{Steiner symmetric} if $\Omega^\sharp=\Omega$.
\end{definition}

It can be shown that $|\Omega|=|\Omega^\sharp|$.

\begin{definition}
Let $\Omega\subset\mathbb{R}^N$ be a measurable set of finite measure and
$u: \Omega \to \mathbb{R}$ a measurable function bounded from below. Then\\
i) the function $u^\sharp: \Omega^\sharp \to \mathbb{R}$, defined by 
\begin{equation*}
u^\sharp(x)=\sup\{c\in \mathbb{R}:x\in\{u>c\}^\sharp\}
\end{equation*}
is said \emph{Steiner symmetrization} of $u$ with respect to the hyperplane $T$;\\
ii) the function $u$ is said  \emph{Steiner symmetric} if $u^\sharp=u$.
\end{definition}

It can be proved that 
\begin{equation}\label{dino}
|\{x\in \Omega^\sharp: u^\sharp(x)>t\}|=|\{x\in \Omega: u(x)>t\}|\quad \forall\, t\in \mathbb{R}.
\end{equation}

\begin{proposition}\label{s1}
Let $\Omega\subset\mathbb{R}^N$ be a measurable set of finite measure, 
$u: \Omega \to \mathbb{R}$ a measurable function bounded from below and $\psi: \mathbb{R}\to\mathbb{R}$ an increasing
function. Then $\psi(u^\sharp)= (\psi(u))^\sharp$ a.e. in $\Omega$.
\end{proposition}

For the proof see \cite[Lemma 3.2]{Bro}.

\begin{proposition}[Hardy-Littlewood's inequality]\label{s2}
Let $\Omega\subset\mathbb{R}^N$ be a measurable set of finite measure, $u, v:\Omega \to \mathbb{R}$ 
two measurable functions bounded from below such that $uv\in L^1(\Omega)$.  Then
\begin{equation*}
\int_\Omega u(x)v(x)\,dx\leq \int_{\Omega^\sharp}u^\sharp(x)v^\sharp(x)\,dx.
\end{equation*}\end{proposition}

This proposition follows easily from \cite[Lemma 3.3]{Bro}.

\begin{proposition}[nonlocal P\`olya-Szeg\"o's inequality]\label{s3}
Let $\Omega\subset\mathbb{R}^N$ be an open bounded set, $s\in (0,1)$ and 
$u\in H_0^s(\Omega)$. Then
\begin{equation*}\label{2b}
\int_{\mathbb{R}^{2N}} \cfrac{|u^\sharp(x)-u^\sharp(y)|^2}{|x-y|^{N+2s}}\,dx\, dy\leq 
\int_{\mathbb{R}^{2N}} \cfrac{|u(x)-u(y)|^2}{|x-y|^{N+2s}}\,dx\, dy;
\end{equation*}
moreover, the equality holds if and only if $u$ is proportional to a translate of a 
function which is symmetric with respect to the hyperplane 
$T=\{x=(x_1,x')\in \mathbb{R}^N: x_1=0\}$.
\end{proposition}

For the proof we refer the reader to \cite{Contador}.

\begin{theorem}\label{T1}
Let $\Omega\subset\mathbb{R}^N$ be a bounded domain of class $C^{2}$ Steiner 
symmetric with respect to the hyperplane
$T=\{x=(x_1,x')\in \mathbb{R}^N: x_1=0\}$ and $\rho_0\in L^\infty(\Omega)$ such that 
$|\{\rho_0>0\}|>0$. 
Then, every minimizer $\check{\rho}_1$ of the problem \eqref{inf} is Steiner 
symmetric relative to $T$. 
\end{theorem}

\begin{proof}
Let $\check{\rho}_1$ be as in \eqref{inf} and 
 let $u_{\check{\rho}_1}$ be the 
positive first eigenfunction of the problem \eqref{p0} normalized as in \eqref{normaliz1}.

By \eqref{carat} and Proposition \ref{s1}, the Steiner symmetry of $\check{\rho}_1$ 
is a consequence of the analogous symmetry of $u_{\check{\rho}_1}$; hence it suffices to show that $u_{\check{\rho}_1}^\sharp=u_{\check{\rho}_1}$. By \eqref{2a} we have
\begin{equation*}
\lambda_1(\check{\rho}_1)=\cfrac{\|u_{\check{\rho}_1}\|^2_{H_0^s(\Omega)}}
{\int_\Omega \check{\rho}_1 u_{\check{\rho}_1}^2\,dx}\, .
\end{equation*}
Propositions \ref{s1}, \ref{s2}  and \ref{s3} yield
\begin{equation*}
 \int_\Omega \check{\rho}_1 u_{\check{\rho}_1}^2\,dx\leq
 \int_\Omega \check{\rho}_1^\sharp (u_{\check{\rho}_1}^2)^\sharp\,dx
 = \int_\Omega \check{\rho}_1^\sharp (u_{\check{\rho}_1}^\sharp)^2\,dx
\end{equation*}
and
\begin{equation}\label{elefante}
\|u_{\check{\rho}_1}\|^2_{H_0^s(\Omega)}
\geq \|u_{\check{\rho}_1}^\sharp\|^2_{H_0^s(\Omega)}.
\end{equation}
Consequently we find
\begin{equation*}
\lambda_1(\check{\rho}_1)=\cfrac{\|u_{\check{\rho}_1}\|^2_{H_0^s(\Omega)}}
{\int_\Omega \check{\rho}_1 u_{\check{\rho}_1}^2\,dx}\,
\geq\cfrac{\|u_{\check{\rho}_1}^\sharp\|^2_{H_0^s(\Omega)}}
{\int_\Omega \check{\rho}_1^\sharp (u_{\check{\rho}_1}^\sharp)^2\,dx}\,\geq
\cfrac{{\|u_{\check{\rho}_1^\sharp}\|^2_{H_0^s(\Omega)}}}
{\int_\Omega \check{\rho}_1^\sharp u_{\check{\rho}_1^\sharp}^2\,dx}
=\lambda_1(\check{\rho}_1^\sharp)\geq \lambda_1(\check{\rho}_1),
\end{equation*}
where $u_{\check{\rho}_1^\sharp}$ is the normalized positive first 
eigenfunction corresponding to $\check{\rho}_1^\sharp$ and the last inequality comes from 
$\check{\rho}_1^\sharp\in \mathcal{G}(\rho_0)$ (a straightforward consequence of \eqref{dino}) and 
the minimality of $\check{\rho}_1$.
Therefore, all the inequalities are actually equalities and this implies the equality sign also
in \eqref{elefante}. Then, by Proposition \ref{s3} it follows that
\begin{equation*}u_{\check{\rho}_1}(x)=\widetilde{u}(x),
\end{equation*}
where $\widetilde{u}$ is Steiner symmetric with respect to a hyperplane 
$T_v=\{x=(x_1,x')\in \mathbb{R}^N: x_1=v\}$, $v\in \mathbb{R}$.
Therefore $\Omega =\{u_{\check{\rho}_1}>0\}=\{\widetilde{u}>0\}$ is symmetric with respect 
to both hyperplanes $T$ and $T_v$.
Being $\Omega$ bounded, it follows that $v=0$ and then 
$u_{\check{\rho}_1}$ is Steiner symmetric 
relative to $T$, i.e. 
\begin{equation*}\label{finale}
u_{\check{\rho}_1}^\sharp=u_{\check{\rho}_1}.
\end{equation*}
This completes the proof.
\end{proof}

In particular, when $\Omega$ is a ball we find the following assertion.

\begin{corollary}
Let $\Omega$ be a ball in $\mathbb{R}^N$ and $\rho_0\in L^\infty(\Omega)$ such that 
$|\{\rho_0>0\}|>0$. Then every minimizer $\check{\rho}_1$
of problem \eqref{inf} is decreasing 
radially symmetric.
\end{corollary}

\begin{rem}
Note that, in this case, $\check{\rho}_1$ is unique and explicitly determined by the class of 
rearrangement of $\rho_0$. Indeed, we have $\check{\rho}_1(x)=\rho_0^*(\omega_N |x|^N)$
for any $x\in \Omega$, where $\omega_N$ denotes the measure of the unit ball in 
$\mathbb{R}^N$.
\end{rem}

Recalling Remark \ref{ape}, we can immediately state the symmetry results for 
$\lambda_{-1}(\rho)$. 

\begin{theorem}
Let $\Omega\subset\mathbb{R}^N$ be a bounded domain of class $C^{2}$ Steiner 
symmetric with respect to the hyperplane
$T=\{x=(x_1,x')\in \mathbb{R}^N: x_1=0\}$ and $\rho_0\in L^\infty(\Omega)$ such that 
$|\{\rho_0<0\}|>0$. 
Then, every maximizer $\check{\rho}_{-1}$ of the problem \eqref{inf2} is such that $-\check{\rho}_{-1}$
is Steiner symmetric relative to $T$. 
\end{theorem}

\begin{corollary}
Let $\Omega$ be a ball in $\mathbb{R}^N$ and $\rho_0\in L^\infty(\Omega)$ such that 
$|\{\rho_0<0\}|>0$. Then every maximizer $\check{\rho}_{-1}$
of problem \eqref{inf2} is increasing 
radially symmetric. More precisely, we have the unique maximizer $\check{\rho}_{-1}(x)=
-(-\rho_0)^*(\omega_N |x|^N)$
for any $x\in \Omega$.
\end{corollary}

\textbf{Acknowledgement}.
The authors are members of GNAMPA (Gruppo Nazionale per l'Analisi Matematica, la Probabilit\`a e le loro Applicazioni) of INdAM (Istituto Nazionale di Alta Matematica ``Francesco Severi").
Claudia Anedda and Fabrizio Cuccu are partially supported by the research project Integro-differential Equations and Non-Local Problems, funded by Fondazione di Sardegna (2017).

\tt Claudia Anedda: canedda@unica.it\\
Fabrizio Cuccu: fcuccu@unica.it\\
Silvia Frassu: silvia.frassu@unica.it

\end{document}